\documentclass[11pt,reqno]{amsart}
\usepackage{amsmath, amsthm, amscd, amsfonts, amssymb, graphicx, color}
\usepackage[bookmarksnumbered, colorlinks, plainpages]{hyperref}

\newtheorem{theorem}{Theorem}[section]
\newtheorem{lemma}[theorem]{Lemma}

\theoremstyle{definition}
\newtheorem{definition}[theorem]{Definition}

\theoremstyle{remark}
\newtheorem{remark}[theorem]{Remark}
\newtheorem{note}[theorem]{Note}
\numberwithin{equation}{section}

\newcommand{\be}{\begin{equation}}
\newcommand{\ee}{\end{equation}}
\newcommand{\bea}{\begin{eqnarray}}
\newcommand{\eea}{\end{eqnarray}}
\begin{document}

\begin{center}
{\large{\textbf{$N(k)$-contact metric as $\ast$-conformal Ricci soliton}}}

\end{center}
\vspace{0.1 cm}

\begin{center}

Dibakar Dey and Pradip Majhi\\
Department of Pure Mathematics,\\
University of Calcutta,
35 Ballygunge Circular Road,\\
Kolkata - 700019, West Bengal, India,\\
E-mail: deydibakar3@gmail.com, mpradipmajhi@gmail.com\\
\end{center}

\vspace{0.3 cm}
\textbf{Abstract:} The aim of this paper is characterize a class of contact metric manifolds admitting $\ast$-conformal Ricci soliton. It is shown that if a $(2n + 1)$-dimensional $N(k)$-contact metric manifold $M$ admits $\ast$-conformal Ricci soliton or $\ast$-conformal gradient Ricci soliton, then the manifold $M$ is $\ast$-Ricci flat and locally isometric to the Riemannian of a flat $(n + 1)$-dimensional manifold and an $n$-dimensional manifold of constant curvature $4$ for $n > 1$ and flat for $n = 1$. Further, for the first case, the soliton vector field is conformal and for the $\ast$-gradient case, the potential function $f$ is either harmonic or satisfy a Poisson equation. Finally, an example is presented to support the results.



\section{ \textbf{Introduction}}
\vspace{0.3 cm}

In 2004, Fischer \cite{fischer} introduced the notion of conformal Ricci flow as a variation of the classical Ricci flow equation. Let $M$ be an $n$-dimensional closed, connected, oriented differentiable manifold. Then the conformal Ricci flow on $M$ is defined by
\bea
\nonumber \frac{\partial g}{\partial t} + 2(S + \frac{g}{n}) = - pg \;\; \mathrm{and} \;\; r = - 1,
\eea
where $p$ is a time dependent non-dynamical scalar field, $S$ is the $(0,2)$ symmetric Ricci tensor and $r$ is the scalar curvature of the manifold.\\

\noindent The concept of conformal Ricci soliton was introduced by Basu and Bhattacharyya \cite{bb} on a $(2n + 1)$-dimensional Kenmotsu manifold as
\bea
\nonumber \pounds_V g + 2S = [2\lambda - (p + \frac{2}{2n + 1})]g,
\eea
where $\lambda$ is a constant and $\pounds_V$ is the Lie derivative along the vector field V. This notion was studied by Dey and Majhi \cite{dey1}, Nagaraja and Venu \cite{hg} and many others on several contact metric manifolds.\\

\noindent In 2002, Hamada \cite{hamada} defined the $\ast$-Ricci tensor on real hypersurfaces of complex space forms by
\bea
\nonumber S^\ast(X,Y) = g(Q^\ast X,Y) = \frac{1}{2}(trace(\phi \circ R(X,\phi Y)))
\eea
for any vector fields $X$, $Y$ on $M$, where $Q^\ast$ is the $(1,1)$ $\ast$-Ricci operator. The $\ast$-scalar curvature $r^\ast$ is defined by $r^\ast = trace(Q^\ast)$. A Riemannian manifold $M$ is called $\ast$-Ricci flat if $S^\ast$ vanishes identically.\\

\noindent Recently, several notions related to the $\ast$-Ricci tensor were introduced. In 2014, the notion of $\ast$-Ricci soliton \cite{kp} was introduced and further widely studied by several authors. In 2019, the notion of $\ast$-critical point equation \cite{dey2} was introduced and further studied by the authors in \cite{dey3}. In this paper, we study the notion of $\ast$-conformal Ricci soliton defined as
\begin{definition}
A Riemannian manifold $(M,g)$ of dimension $(2n + 1) \geq 3$ is said to admit $\ast$-conformal Ricci soliton $(g,V,\lambda)$ if
 \bea
\pounds_V g + 2S^\ast = [2\lambda - (p + \frac{2}{2n + 1})]g, \label{1.1}
\eea
where $\lambda$ is a constant, provided $S^\ast$ is symmetric.
\end{definition}
\noindent A $\ast$-conformal Ricci soliton is said to be $\ast$-conformal gradient Ricci soliton if the vector field $V$ is gradient of some smooth function $f$ on $M$. In this case, the $\ast$-conformal gradient Ricci soliton is given by
\bea
\nabla^2 f  + S^\ast = [\lambda - (\frac{p}{2} + \frac{1}{2n + 1})]g, \label{1.2}
\eea
where $(\nabla^2 f)(X,Y) = Hess f (X,Y) = g(\nabla_X Df,Y)$ is the Hessian of $f$ and $D$ is the gradient operator.\\

\noindent Note that, the $\ast$-Ricci tensor is not symmetric in general. Hence, for a non-symmetric $\ast$-Ricci tensor of a manifold, the above notion is incosistant. In a $N(k)$-contact metric manifold, $S^\ast$ is symmetric (given later) and hence, the above definition is well defined on $N(k)$-contact metric manifolds.\\

\noindent The present paper is organized as follows: In section 2, we recall some preliminary results from the literature of $N(k)$-contact metric manifolds. Section 3 deals with $N(k)$-contact metric manifolds admitting $\ast$-conformal Ricci soliton and $\ast$-conformal gradient Ricci soliton. In the final section, we present an example to verify our results.

\section{\textbf{Preliminaries}}

A $(2n + 1)$-dimensional almost contact metric manifold $M$ is a smooth manifold together with a structure $(\phi,\xi,\eta,g)$ satisfying
\bea
\phi^2 X = - X + \eta(X)\xi, \;\; \eta(\xi) = 1, \;\; \phi \xi = 0, \;\; \eta \circ \phi = 0, \label{2.1}
\eea
\bea
g(\phi X,\phi Y) = g(X,Y) - \eta(X)\eta(Y) \label{2.2}
\eea
for any vector fields $X$, $Y$ on $M$, where $\phi$ is a $(1,1)$ tensor field, $\xi$ is a unit vector field, $\eta$ is a one form defined by $\eta(X) = g(X,\xi)$ and $g$ is the Riemannian metric. Using (\ref{2.2}), we can easily see that $\phi$ is skew-symmetric, that is,
\bea
g(\phi X,Y) = - g(X,\phi Y). \label{2.3}
\eea
An almost contact metric structure becomes a contact metric structure if $g(\phi X,Y) = d\eta(X,Y)$ for all vector fields $X$, $Y$ on $M$. On a contact metric manifold, the $(1,1)$-tensor field $h$ is defined as $h = \frac{1}{2} \pounds_\xi \phi $. The tensor field $h$ is symmetric and satisfies
\bea
h\phi = - \phi h, \;\; trace(h) = trace(\phi h) = 0,\;\; h\xi = 0. \label{2.4}
\eea
Also on a contact metric manifold, we have
\bea
\nabla_X \xi = - \phi X - \phi hX. \label{2.5}
\eea
In \cite{tanno}, Tanno introduced the notion of $k$-nullity distribution on a Riemannian manifold as
\bea
\nonumber N(k) = \{ Z \in T(M) : R(X,Y)Z = k[g(Y,Z)X - g(X,Z)Y]\},
\eea
$k$ being a real number and $T(M)$ is the Lie algebra of all vector fields on $M$. If the characteristic vector field $\xi \in N(k)$, then we call a contact metric manifold as $N(k)$-contact metric manifold \cite{tanno}. However, for a $(2n + 1)$-dimensional $N(k)$-contact metric manifold, we have ( see \cite{blair1}, \cite{blair2})
\bea
h^2 = (k - 1)\phi^2, \label{2.6}
\eea
\bea
R(X,Y)\xi = k[\eta(Y)X - \eta(X)Y], \label{2.7}
\eea
\bea
R(\xi,X)Y = k[g(X,Y)\xi - \eta(Y)X], \label{2.8}
\eea
\bea
(\nabla_X \eta)Y = g(X + hX,\phi Y), \label{2.9}
\eea
\bea
(\nabla_X \phi)Y = g(X + hX,Y)\xi - \eta(Y)(X + hX), \label{2.10}
\eea
\bea
\nonumber (\nabla_X \phi h)Y &=& [g(X,hY) + (k - 1)g(X,-Y + \eta(Y)\xi)]\xi \\ && + \eta(Y)[hX + (k - 1)(- X + \eta(X)\xi)] \label{2.11}
\eea
for any vector fields $X$, $Y$ on $M$, where $R$ is the Riemann curvature tensor. For further details on $N(k)$-contact metric manifolds, we refer the reader to go through thereferences (\cite{avik}, \cite{ozgur1}, \cite{ozgur2}) and references therein.

\section{\textbf{$\ast$-Conformal Ricci soliton}}

In this section, we study the notion of $\ast$-conformal Ricci soliton  in the framework of $N(k)$-contact metric manifolds. To prove the main theorems, we need the following lemmas:

\begin{lemma} (\cite{blair3}) \label{l3.1}
A contact metric manifold $M^{2n+1}$ satisfying the condition $R(X,Y)\xi = 0$ for all $X$, $Y$ is locally isometric to the Riemannian product of a flat $(n + 1)$-dimensional manifold and an $n$-dimensional manifold of positive curvature $4$, i.e., $E^{n+1}(0) \times S^n(4)$ for $n > 1$ and flat for $n = 1$.
\end{lemma}

\begin{lemma} (\cite{dey2}) \label{l3.2}
A $(2n + 1)$-dimensional $N(k)$-contact metric manifold is $\ast$-$\eta$-Einstein and the $\ast$-Ricci tensor is given by
\bea
S^\ast (X,Y) = -k[g(X,Y) - \eta(X)\eta(Y)]. \label{3.1}
\eea
\end{lemma}

\begin{note}
We observe from lemma \ref{l3.2} that the $\ast$-Ricci tensor $S^\ast$ of a $N(k)$-contact metric manifold is symmetric, i.e., $S^\ast(X,Y) = S^\ast(Y,X)$. Hence, the notion of $\ast$-conformal Ricci soliton is consistent in this setting.
\end{note}

\begin{lemma}\label{l3.4}
On a $(2n + 1)$-dimensional $N(k)$-contact metric manifold $M$, the $\ast$-Ricci tensor $S^\ast$ satisfies the following relation:
\bea
\nonumber && (\nabla_Z S^\ast)(X,Y) - (\nabla_X S^\ast)(Y,Z) - (\nabla_Y S^\ast)(X,Z) \\ \nonumber &&= - 2k[\eta(Y)g(\phi X,Z) + \eta(Y)g(\phi X,hZ) + \eta(X)g(\phi Y,Z) + \eta(X)g(\phi Y,hZ)]
\eea
for any vector fields $X$, $Y$ and $Z$ on $M$.
\end{lemma}
\begin{proof}
Differentiating (\ref{3.1}) covariantly along any vector field $Z$, we have
\bea
 \nabla_Z S^\ast(X,Y) = -k[\nabla_Z g(X,Y) -(\nabla_Z\eta(X))\eta(Y)-(\nabla_Z\eta(Y))\eta(X)]. \label{3.2}
\eea
Now,
\bea
\nonumber (\nabla_Z S^\ast)(X,Y)=\nabla_Z S^\ast(X,Y)-S^\ast(\nabla_Z X,Y)-S^\ast(X,\nabla_ZY).
\eea
Using (\ref{3.1}) and (\ref{3.2}) in the foregoing equation, we obtain
\bea\label{3.3}
(\nabla_Z S^\ast)(X,Y) = k[((\nabla_Z\eta)X)\eta(Y)+((\nabla_Z\eta)Y)\eta(X)]. \label{3.3}
\eea
Using (\ref{2.9}) in (\ref{3.3}), we infer that
\bea
(\nabla_Z S^\ast)(X,Y) =k[\eta(Y)g(Z + hZ,\phi X) + \eta(X)g(Z + hZ,\phi Y)].  \label{3.4}
\eea
In a similar manner, we get
\bea
(\nabla_X S^\ast)(Y,Z) = k[\eta(Z)g(X + hX,\phi Y) + \eta(Y)g(X + hX,\phi Z)]. \label{3.5}
\eea
\bea
(\nabla_Y S^\ast)(X,Z) = k[\eta(Z)g(Y + hY,\phi X) + \eta(X)g(Y + hY,\phi Z)]. \label{3.6}
\eea
With the help of (\ref{3.4})-(\ref{3.6}), we complete the proof by using (\ref{2.3}) and (\ref{2.4}).
\end{proof}
\begin{theorem} \label{t3.5}
 Let $M$ be a $(2n + 1)$-dimensional
 $N(k)$-contact metric manifold admitting $\ast$-conformal Ricci soliton $(g,V,\lambda)$, then
 \begin{itemize}
 \item[$(1)$] The manifold $M$ is locally isometric to $E^{n+1}(0) \times S^n(4)$ for $n > 1$ and flat for $n = 1$.
 \item[$(2)$] The manifold $M$ is $\ast$-Ricci flat.
 \item[$(3)$] The vector field V is conformal,
 \end{itemize}
 provided $\lambda \neq \frac{p}{2} + \frac{1}{2n + 1}$.
\end{theorem}

\begin{proof}
From (\ref{1.1}), we have
 \bea
 (\pounds_V g)(X,Y) + 2S^\ast(X,Y) = [2\lambda - (p + \frac{2}{2n+1})]g(X,Y). \label{3.7}
 \eea
Differentiating the above equation covariantly along any vector field $Z$, we get
\bea
(\nabla_Z \pounds_V g)(X,Y) = - 2(\nabla_Z S^\ast)(X,Y). \label{3.8}
\eea
It is well known that (see \cite{yano})
\be
\nonumber (\pounds_V \nabla_X g - \nabla_X \pounds_V g - \nabla_{[V,X]}g)(Y,Z) = -g((\pounds_V \nabla)(X,Y),Z) - g((\pounds_V \nabla)(X,Z),Y).
\ee
Since $\nabla g = 0$, then the above relation becomes
 \bea
(\nabla_X \pounds_V g)(Y,Z) = g((\pounds_V \nabla)(X,Y),Z) + g((\pounds_V \nabla)(X,Z),Y). \label{3.9}
\eea
Since $\pounds_V \nabla$ is symmetric, then it follows from (\ref{3.9}) that
\bea
\nonumber g((\pounds_V \nabla)(X,Y),Z) &=& \frac{1}{2}(\nabla_X \pounds_V g)(Y,Z) + \frac{1}{2}(\nabla_Y \pounds_V g)(X,Z) \\ && - \frac{1}{2}(\nabla_Z \pounds_V g)(X,Y). \label{3.10}
\eea
Using (\ref{3.8}) in (\ref{3.10}) we have
\be
\nonumber g((\pounds_V \nabla)(X,Y),Z) = (\nabla_Z S^\ast)(X,Y) - (\nabla_X S^\ast)(Y,Z) - (\nabla_Y S^\ast)(X,Z).
\ee
Now, using Lemma \ref{l3.4} in the foregoing equation yields
\bea
\nonumber g((\pounds_V \nabla)(X,Y),Z) &=& 2k[\eta(Y)g(\phi X,Z) + \eta(Y)g(\phi X,hZ) \\ \nonumber &&+ \eta(X)g(\phi Y,Z) + \eta(X)g(\phi Y,hZ)],
\eea
which implies
\bea
 (\pounds_V \nabla)(X,Y) = 2k[\eta(Y)\phi X + \eta(Y)h\phi X + \eta(X)\phi Y + \eta(X)h\phi Y]. \label{3.11}
\eea
Substituting $Y = \xi$ in (\ref{3.11}), we get
\bea
(\pounds_V \nabla)(X,\xi) = 2k[\phi X + h\phi X].\label{3.12}
\eea
Differentiating (\ref{3.12}) along any vector field $Y$, we obtain
\bea
\nabla_Y (\pounds_V \nabla)(X,\xi) = 2k[\nabla_Y \phi X + \nabla_Y h\phi X]. \label{3.13}
\eea
Now,
\bea
\nonumber (\nabla_Y \pounds_V \nabla)(X,\xi) = \nabla_Y (\pounds_V \nabla)(X,\xi) - (\pounds_V \nabla)(\nabla_Y X,\xi) - (\pounds_V \nabla)(X,\nabla_Y \xi).
\eea
Using (\ref{2.1})-(\ref{2.5}) and (\ref{3.11})-(\ref{3.13}) in the foregoing equation, we obtain
\bea
\nonumber (\nabla_Y \pounds_V \nabla)(X,\xi) &=& 2k[(\nabla_Y \phi)X + (\nabla_Y h\phi)X + (k - 2)\eta(X)\xi \\ && - (k - 2)\eta(X)\eta(Y)\xi - 2\eta(X)hY]. \label{3.14}
\eea
Now, using (\ref{2.10}) and (\ref{2.11}) in (\ref{3.14}), we get
\bea
\nonumber (\nabla_Y \pounds_V \nabla)(X,\xi) &=& 2k[kg(X,Y)\xi + (k - 2)\eta(Y)X \\ \nonumber && + (k - 2)\eta(X)Y  - 2\eta(X)hY \\ && - 2\eta(Y)hX - (3k - 4)\eta(X)\eta(Y)\xi]. \label{3.15}
\eea
Due to Yano \cite{yano}, it is known that
\bea
\nonumber (\pounds_V R)(X,Y)Z = (\nabla_X \pounds_V \nabla)(Y,Z) - (\nabla_Y \pounds_V \nabla)(X,Z),
\eea
Using the equation (\ref{3.15}) in the above formula, we obtain
\bea
(\pounds_V R)(X,\xi)\xi = (\nabla_X \pounds_V \nabla)(\xi,\xi) - (\nabla_\xi \pounds_V \nabla)(X,\xi) = 0. \label{3.16}
\eea
Now, substituting $Y = \xi$ in (\ref{3.7}), we have
\bea
(\pounds_V g)(X,\xi)  = [2\lambda - (p + \frac{2}{2n+1})]\eta(X), \label{3.17}
\eea
which implies
\bea
(\pounds_V \eta)X - g(X,\pounds_V \xi) = [2\lambda - (p + \frac{2}{2n+1})]\eta(X). \label{3.18}
\eea
From (\ref{3.18}), after putting $X = \xi$, we can easily obtain that
\bea
\eta(\pounds_V \xi) = -[\lambda - (\frac{p}{2} + \frac{1}{2n+1})]. \label{3.19}
\eea
Now, from (\ref{2.7}), we have
\bea
 R(X,\xi)\xi = k(X - \eta(X)\xi). \label{3.20}
\eea
With the help of (\ref{3.18})-(\ref{3.20}) and (\ref{2.7})-(\ref{2.8}), we obtain
\bea
 (\pounds_V R)(X,\xi)\xi = k[2\lambda - (p + \frac{2}{2n+1})](X - \eta(X)\xi). \label{3.21}
\eea
Equating (\ref{3.16}) and (\ref{3.21}) and then taking inner product with $Y$ yields
\bea
\nonumber  k[2\lambda - (p + \frac{2}{2n+1})](g(X,Y) - \eta(X)\eta(Y)) = 0,
\eea
which implies $k = 0$, since by hypothesis, $ \lambda \neq \frac{p}{2} + \frac{1}{2n + 1}$. Therefore, from (\ref{3.1}), we have $S^\ast = 0$, i.e., the manifold is $\ast$-Ricci flat. Again from (\ref{2.7}), we have $R(X,Y)\xi = 0$ and hence, from lemma \ref{l3.1}, it follows that the manifold $M$ is locally isometric to $E^{n+1}(0) \times S^n(4)$ for $n > 1$ and flat for $n = 1$.\\
Now, we know that a vector field $X$ on a Riemannian manifold $M$ is said to be conformal if there is a smooth function $\sigma$ on $M$ such that $\pounds_X g = 2\sigma g$. Using $S^\ast = 0$ in (\ref{3.7}), we get $\pounds_V g = 2[\lambda - (\frac{p}{2} + \frac{1}{2n + 1})]g$ and hence $V$ is a conformal vector field.
\end{proof}

\begin{note}
If $\lambda= (\frac{p}{2} + \frac{1}{2n+1})$, then from (\ref{1.1}), we can say that the $\ast$-conformal Ricci soliton reduces to a steady $\ast$-Ricci soliton. We need the following well known definition to discuss about this further.
\end{note}

\begin{definition}
On an almost contact metric manifold $M$, a vector field $V$ is said to be Killing if $\pounds_Vg=0$ and an infinitesimal contact transformation if $\pounds_V\eta=f\eta$ for some smooth function $f$ on $M$. In particular, if $f=0$, then $V$ is said to be strict infinitesimal contact transformation.
\end{definition}

\begin{remark}
If $k \neq 0$ and $\lambda = (\frac{p}{2} + \frac{1}{2n+1})$, then from (\ref{3.18}), we have $(\pounds_V\eta)X=g(X,\pounds_V\xi)$. Thus $V$ will be an infinitesimal contact transformation if $\pounds_V\xi=f\xi$ for some smooth function $f$ on $M$. But in view of (\ref{3.19}), we have $\eta(\pounds_V\xi)=0$, which implies $\pounds_V\xi \bot \xi$. Hence $\pounds_V\xi\not=f\xi$ for any smooth function $f$ on $M$, unless $f=0$ identically. Hence, $V$ cannot be an infinitesimal contact transformation on $M$ but it can be a strict infinitesimal contact transformation if $\pounds_V \xi = 0$.
\end{remark}

\begin{remark} \label{r3.9}
 If $k = 0$ and $\lambda= (\frac{p}{2} + \frac{1}{2n+1})$, then from (\ref{3.7}), we have $\pounds_Vg = 0$. Hence $V$ is a Killing vector field.
\end{remark}

To prove our next theorem regarding $\ast$-conformal gradient Ricci soliton, we first state and prove the following lemma:
\begin{lemma} \label{l3.10}
Let $M$ be $(2n + 1)$-dimensional $N(k)$-contact metric manifold admitting $\ast$-conformal gradient Ricci soliton $(g,V,\lambda)$. Then the curvature tensor $R$ can be expressed as
\be
R(X,Y)Df = k[2g(\phi X,Y)\xi - \eta(X)(\phi Y + \phi hY) + \eta(Y)(\phi X + \phi hX)] \label{3.22}
\ee
for any vector fields $X$ and $Y$ on $M$, where $ V = Df$.
\end{lemma}
\begin{proof}
Equation (\ref{1.2}) can be written as
\bea
\nabla_X Df = [\lambda - (\frac{p}{2} + \frac{1}{2n + 1 })]X - Q^\ast X. \label{3.23}
\eea
Differentiating (\ref{3.23}) along any vector field $Y$, we obtain
\bea
\nabla_Y \nabla_X Df = [\lambda - (\frac{p}{2} + \frac{1}{2n + 1})]\nabla_Y X - \nabla_Y Q^\ast X. \label{3.24}
\eea
Interchanging $X$ and $Y$ in the above equation, we get
\bea
\nabla_X \nabla_Y Df = [\lambda - (\frac{p}{2} + \frac{1}{2n + 1})]\nabla_X Y - \nabla_X Q^\ast Y. \label{3.25}
\eea
Again from (\ref{3.23}), we have
\bea
\nabla_{[X,Y]} Df = [\lambda - (\frac{p}{2} + \frac{1}{2n + 1})](\nabla_X Y - \nabla_Y X) - Q^\ast(\nabla_X Y - \nabla_Y X). \label{3.26}
\eea
It is well known that
\bea
\nonumber R(X,Y)Df = \nabla_X \nabla_Y Df - \nabla_Y \nabla_X Df - \nabla_{[X,Y]}Df.
\eea
Substituting (\ref{3.24})-(\ref{3.26}) in the foregoing equation, we obtain
\bea
R(X,Y)Df = (\nabla_Y Q^\ast)X - (\nabla_X Q^\ast)Y. \label{3.27}
\eea
Now, from (\ref{3.5}) and (\ref{3.6}), we can write
\bea
(\nabla_X Q^\ast)Y = k[g(X + hX,\phi Y)\xi - \eta(Y)(\phi X + \phi hX)]. \label{3.28}
\eea
\bea
(\nabla_Y Q^\ast)X = k[g(Y + hY,\phi X)\xi - \eta(X)(\phi Y + \phi hY)]. \label{3.29}
\eea
We now complete the proof by substituting (\ref{3.28}) and (\ref{3.29}) in (\ref{3.27}).
\end{proof}
\begin{theorem} \label{t3.11}
 Let $M$ be a $(2n + 1)$-dimensional $N(k)$-contact metric manifold admitting $\ast$-conformal gradient Ricci soliton $(g,V,\lambda)$, where $V = Df$ for some smooth function $f$ on $M$, then
 \begin{itemize}
     \item[$(1)$] The Manifold $M$ is locally isometric to $E^{n+1}(0) \times S^n(4)$ for $n > 1$ and flat for $n = 1$.
     \item[$(2)$] The manifold $M$ is $\ast$-Ricci flat.
     \item[$(3)$] The potential function $f$ is either harmonic or satisfy a physical Poisson equation.
 \end{itemize}
\end{theorem}

\begin{proof}
Putting $X = \xi$ in (\ref{3.22}), we obtain
\bea
\nonumber R(\xi,Y)Df = - k[\phi Y + \phi hY].
\eea
Taking inner product of the foregoing equation yields
\bea
g(R(\xi,Y)Df,X) = - k[g(\phi X,Y) + g(\phi hX,Y)]. \label{3.30}
\eea
Since $g(R(\xi,Y)Df,X) = - g(R(\xi,Y)X,Df)$, then using (\ref{2.8}), we obtain
\bea
g(R(\xi,Y)Df,X) = - kg(X,Y)(\xi f) + k\eta(X)(Yf). \label{3.31}
\eea
Equating (\ref{3.30}) and (\ref{3.31}) and then antisymmetrizing yields
\bea
\nonumber k\eta(X)(Yf) - k\eta(Y)(Xf) - 2kg(\phi X,Y) = 0.
\eea
Putting $X = \xi$ in the above equation, we obtain
\bea
\nonumber k[(Yf) - (\xi f)\eta(Y)] = 0,
\eea
which implies
\bea
\nonumber k[Df - (\xi f)\xi] = 0.
\eea
Hence, either $k = 0$ or $Df = (\xi f)\xi$.\\

Case 1: If $k = 0$, then (\ref{3.1}) implies $S^\ast = 0$. Also from (\ref{2.7}), $R(X,Y)\xi = 0$. Using lemma \ref{l3.1}, we can say that $M$ is locally isometric to $E^{n+1}(0) \times S^n(4)$ for $n > 1$ and flat for $n = 1$. Again using $S^\ast = 0$ in (\ref{1.2}) and then tracing yields $\Delta f = [\lambda - (\frac{p}{2} + \frac{1}{2n + 1})](2n + 1)$, where $\Delta$ is the Laplace operator. Hence, $f$ satisfies a
Poisson equation.\\

Case 2: If $Df = (\xi f)\xi$, then differentiating this along any vector field $X$, we obtain
\bea
\nabla_X Df = (X(\xi f))\xi + (\xi f)(- \phi X - \phi hX).  \label{3.32}
\eea
Equating (\ref{3.23}) and (\ref{3.32}), we get
\bea
Q^\ast X = [\lambda - (\frac{p}{2} + \frac{1}{2n + 1})]X - (X(\xi f))\xi + (\xi f)(\phi X + \phi hX). \label{3.33}
\eea
Comparing the coefficients of $X$, $\xi$ and $\phi X$ from (\ref{3.1}) and (\ref{3.33}), we obtain the followings
\bea
\lambda - (\frac{p}{2} + \frac{1}{2n + 1}) = - k. \label{3.34}
\eea
\bea
X(\xi f) = k\eta(X). \label{3.35}
\eea
\bea
(\xi f) = 0. \label{3.36}
\eea
Using (\ref{3.36}) in (\ref{3.35}), we obtain $k = 0$. Hence, from (\ref{3.34}), $\lambda = \frac{p}{2} + \frac{1}{2n + 1}$. Since $k = 0$, by the same argument as in case 1, the manifold $M$ is $\ast$-Ricci flat and locally isometric to $E^{n+1}(0) \times S^n(4)$ for $n > 1$ and flat for $n = 1$. Again using $S^\ast = 0$ and $\lambda = \frac{p}{2} + \frac{1}{2n + 1}$ in (\ref{1.2}), we obtain $\nabla^2 f = 0$, which implies $\Delta f = 0$. Therefore, $f$ is harmonic. This completes the proof.
\end{proof}

\section{Example}
In \cite{ucd}, the authors presented an example of a $3$-dimensional $N(1 - \alpha^2)$-contact metric manifold. Using the expressions of the curvature tensor and several values of the linear connection, we can easily calculate the followings:\\
$$S^\ast(e_1,e_1) = 0, \;\; S^\ast(e_2,e_2) = S^\ast(e_3,e_3) = - (1 - \alpha^2).$$
$$(\pounds_{e_1} g)(X,Y) = 0 \;\; \mathrm{for\; all} \;\; X,\; Y \in \{e_1, e_2, e_3\}.$$
Now, if we consider $\alpha = 1$, then the curvature tensor $R$ vanishes and also $S^\ast = 0$. Tracing (\ref{1.1}), we get $\lambda = \frac{p}{2} + \frac{1}{3}$. Thus $(g,e_1,\lambda)$ is a $\ast$-conformal Ricci soliton on this $N(0)$-contact metric manifold. Here, note that $e_1$ is a Killing vector field. This verifies our theorem \ref{t3.5} and remark \ref{r3.9}.\\
Again if $e_1 = Df$ for some smooth function $f$, then tracing (\ref{1.2}) and considering $\alpha = 1$, we have $\Delta f = [\lambda - (\frac{p}{2} + \frac{1}{3})]3$, which is a Poisson equation. Also if $\lambda = \frac{p}{2} + \frac{1}{3}$, then $\Delta f = 0$ and therefore, $f$ is harmonic. This verifies our theorem \ref{t3.11}.\\

\subsection*{Acknowledgements}
The author Dibakar Dey is thankful to the Council of Scientific and Industrial Research, India (File No. 09/028(1010)/2017-EMR-1) for their assistance.


\begin{thebibliography}{99}

 \bibitem{bb} N. Basu and A. Bhattacharyya, {\em Conformal Ricci solition in Kenmotsu manifold}, Glob. J. Adv. Res. Class. Mod. Geom. 4(2015), 15-21.

 \bibitem{blair1} D. E. Blair, {\em Contact manifolds in Riemannian geometry}, Lecture Notes on Mathematics, 509, Springer, Berlin, 1976.

 \bibitem{blair3} D. E. Blair, {\em Two remarks on contact metric structure}, Tohoku Math. J. 29(1977), 319-324.

 \bibitem{blair2} D. E. Blair, {\em Riemannian geometry on contact and symplectic manifolds}, Progress in Mathematics, 203, Birkh$\ddot{a}$user, Boston, 2010.

\bibitem{avik} A. De and J. B. Jun, {\em On $N(k)$-contact metric manifolds satisfying certain curvature conditions}, Kyungpook Math. J. 51(2011), 457-468.

\bibitem{ucd} U. C. De, A. Yildiz and S. Ghosh, {\em On a class of $N(k)$-contact metric manifolds}, Math. Reports 16(2004), 207-217.

 \bibitem{dey1} D. Dey and P. Majhi, {\em Almost Kenmotsu metric as conformal Ricci soliton}, Conform. Geom. Dyn. 23(2019), 105-116.

 \bibitem{dey2} D. Dey and P. Majhi, {\em $\ast$-Critical point equation on $N(k)$-contact manifolds}, Bull. Transliv. Univ. Brasov. Ser. III 12(2019), 275-282.

 \bibitem{dey3} D. Dey and P. Majhi, {\em $\ast$-Critical point equation on a class of almost Kenmotsu manifolds}, J. Geom. 111(2020), no. 1, paper no. 16.


 \bibitem{fischer} A. E. Fischer, {\em An introduction to conformal Ricci flow}, Class. Quantum Grav. 21(2004), 171-218.

 \bibitem{hamada} T. Hamada {\em Real hypersurfaces of complex space forms in terms of Ricci $\ast$-tensor}, Tokyo J. Math. 25(2002), 473-483.

\bibitem{kp} G. Kaimakanois and K. Panagiotidou, {\em $\ast$-Ricci solitons of real hypersurfaces in non-flat complex space forms}, J. Geom. Phys. 86(2014), 408-413.

\bibitem{hg} H. G. Nagaraja and K. Venu, {\em $f$-Kenmotsu metric as conformal Ricci soliton}, An. Univ. Vest. Timis. Ser. Mat.-Inform. 55(2017), 119-127.

 \bibitem{ozgur1} C. $\ddot{O}$zg$\ddot{u}r$, {\em Contact metric manifolds with cyclic-parallel Ricci tensor}, Diff. Geom. Dynamical systems, 4(2002), 21-25.

\bibitem{ozgur2} C. $\ddot{O}$zg$\ddot{u}$r and S. Sular, {\em On $N(k)$-contact metric manifolds satisfying certain curvature conditions}, SUT J. Math. 44(2008), 89-99.

\bibitem{tanno} S. Tanno, {\em Ricci curvature of contact Riemannian manifolds}, Tohoku Math. J. 40(1988), 441-448.

\bibitem{yano}K. Yano {\em Integral formulas in Riemannian Geometry}, Marcel Dekker, New York, 1970.
\end{thebibliography}
\end{document}